\documentclass[11pt, draft]{amsart}

\usepackage{amscd}
\usepackage{amsmath}
\usepackage{verbatim}

\usepackage{amsfonts, amssymb, color}

\usepackage{url}
\usepackage{hyperref}

\newtheorem{thm}{Theorem}[section]
\newtheorem{question}{Question}[section]
\newtheorem{cor}[thm]{Corollary}
\newtheorem{lem}[thm]{Lemma}
\newtheorem{prop}[thm]{Proposition}

\theoremstyle{definition}

\theoremstyle{remark}

\numberwithin{equation}{section}

\newcommand{\supp}{\mathrm{ supp }}

\newcommand{\vphi}{\varphi}
\newcommand{\wed}{\wedge}
\renewcommand{\[}{\begin{equation}}
\renewcommand{\]}{\end{equation}}

\def\PP{{\mathbb P}}

\def\cC{{\mathcal C}}

\def\cE{{\mathcal E}}

\def\cL{{\mathcal L}}

\def\cU{{\mathcal U}}

\def\psh{\mathrm{PSH}}

\begin{document}
	
\title[Equidistribution of non-pluripolar products]{Equidistribution of non-pluripolar products associated with quasi-plurisubharmonic functions of finite energy}%

\author{Taeyong Ahn}
\author{Ngoc Cuong Nguyen}%
\address{(Ahn) Department of Mathematics Education, Inha University, 100 Inha-ro, Michuhol-gu, Incheon 22212, Republic of Korea}%
\email{t.ahn@inha.ac.kr}

\address{(Nguyen) Faculty of Mathematics and Computer Science, Jagiellonian University 30-348 Krak\'ow, \L ojasiewicza 6, Poland}
\email{Nguyen.Ngoc.Cuong@im.uj.edu.pl}

\date{\today}
\thanks{The first author was supported by INHA UNIVERSITY Research Grant. A part of the work was done while the second author was a postdoc at Postech University, which was financially supported by the NRF Grant 2011-0030044 (SRC-GAIA) of The Republic of Korea. The second is also partially supported by NCN grant 2017/27/B/ST1/01145}%

\begin{abstract}
	In this article, we consider currents given by the $p$-fold non-pluripolar product associated with a quasi-plurisubharmonic function of finite energy, and prove that normalized pull-backs of such currents converge to the Green $(p, p)$-current exponentially fast in the sense of currents for holomorphic endomorphisms of $\PP^k$.
\end{abstract}
\maketitle

\section{Introduction}
In this article, we consider a holomorphic endomorphism $f:\PP^k\longrightarrow\PP^k$ of algebraic degree $\lambda\geq 2$. Let $\omega$ denote the Fubini-Study form of $\PP^k$ normalized so that $\int_{\PP^k}\omega^k=1$ and $T$ the Green current associated with $f$. Let $\cC_p$ denote the set of positive closed $(p, p)$-currents of unit mass on $\PP^k$.
\medskip

In the literature of the equidistribution of the inverse images of analytic subsets or more generally that of positive closed currents for holomorphic endomorphisms of $\PP^k$, the Lelong number has played a crucial rule. More precisely, for $S\in\cC_p$, we are interested in when the normalized pull-backs $S_n:=\lambda^{-p n}(f^n)^*S$ of $S$ converge to the Green $(p, p)$-current $T^p$ associated with $f$. In the case of bidegree $(1, 1)$, if the Lelong number of $S$ vanishes on $\PP^k$, then we have the desired convergence. (e.g. \cite{G}, \cite{DS08}) This property is not known for $p>1$. So, our primary question to study is the following:
\begin{question}\label{q:main}
	Assume that a current $S\in\cC_p$ has zero Lelong number on $\PP^k$. Then, can we show that
	\begin{displaymath}\notag
	\lambda^{-pn}(f^n)^*S\to T^p
	\end{displaymath}
	exponentially fast in the sense of currents?
\end{question}

In the case of $1<p<k$, it is not clear whether the theory of super-potentials works well with the Lelong number. So, we will rather return to use the pluripotential theory and related results. For $u\in\cE^1(\PP^k, \omega)$, the $p$-fold non-pluripolar product $\left<\omega_u^p\right>$ of $\omega_u:=\omega+dd^cu$ has zero Lelong number everywhere on $\PP^k$. We refer the reader to Section \ref{sec:prelim} for the definitions of $\cE^1(\PP^k, \omega)$ and the $p$-fold non-pluripolar product. In this aspect, it is natural to study the equidistribution of non-pluripolar products.
\medskip

The main purpose of this article is to investigate Question \ref{q:main} by proving the following theorem as a model case.  
\begin{thm}\label{thm:main}
	Let $u\in\cE^1(\PP^k, \omega)$. Then, we have for $1\leq p \leq k$,
	\[\notag
	\lambda^{-pn}(f^n)^*\left<\omega_u^p\right>\to T^p
	\]
	exponentially fast in the sense of currents, where $\omega_u:=\omega+dd^cu$ and $$ \left<\omega_u^p\right> = \left<\omega_u \wed \cdots \wed \omega_u\right>$$ 
	is the $p$-fold non-pluripolar product of $\omega_u$.
\end{thm}

In \cite{Ahn}, the first author studied currents in $\cC_p$ admitting super-potentials bounded near an invariant analytic subset for $f$. In general, the currents in Theorem \ref{thm:main} do not admit bounded super-potentials near the set. (cf. Proposition \ref{prop:ex}) In particular, the currents in Theorem \ref{thm:main} may not have bounded/continuous/H\"older continuous super-potentials.
\medskip

The key point in the proof is that positive closed $(1, 1)$-currents with H\"older continuous quasi-potentials and smooth quasi-potentials do not differ much for our estimate. More precisely, we need an estimate of the form
\begin{displaymath}\notag
\int_{\PP^k} u\circ f^n(f^n)^*\langle\omega_u^{p-1}\rangle\wedge \omega^{k-p+1}<C'.
\end{displaymath}
for some constant $C'>0$ independent of $n$. For this, from the invariance of the Green current, we first get an estimate of the form
\begin{displaymath}\notag
\int_{\PP^k} u\circ f^n(f^n)^*\langle\omega_u^{p-1}\rangle\wedge T^{k-p+1}<C
\end{displaymath}
for some constant $C>0$ independent of $n$ and next, we get the desired estimate from the $\cE^1$ condition and the H\"older continuity of the quasi-potential $G$ of $T$. (See Lemma \ref{lem:induction-step} and Corollary \ref{cor:bound-final})
\medskip

	It is also possible to prove the theorem by estimating the volume of sub-level sets of quasi-plurisubharmonic functions as in the case of bidegree $(1, 1)$. In this case, the integrability condition plays the role of the Lelong number vanishing condition. For the simplicity of the presentation, we do not include it here.
\medskip

	Our method does not rely on specific properties of $\PP^k$ and may be used in more general compact K\"ahler manifold settings.

\section{Preliminaries}\label{sec:prelim}

An upper semicontinuous function $u\in\cL^1(\PP^k)$ is said to be $\omega$-pluri-subharmonic ($\omega$-psh) if the current defined by $\omega_u:=\omega+dd^cu$ is positive. We denote by $\psh(\PP^k, \omega)$ the set of $\omega$-psh functions on $\PP^k$.

Note that for any positive closed $(1, 1)$-current $S$ of unit mass (equivalently, cohomologous to $\omega$) on $\PP^k$, there exists a $\psi\in\psh(\PP^k, \omega)$ such that $S=\omega+dd^c\psi$ and that the pull-back operator $f^*$ is well defined for positive closed currents. So, we can write
\[\notag
\frac{1}{\lambda} f^*\omega = \omega + dd^c \vphi,
\]
for some $\vphi\in\psh(\PP^k, \omega)$. By adding a constant, we can normalize $\vphi$ in such a way that $\sup_{\PP^k} \vphi =0$.
It is well-known that
\[\notag
\lim_{N \to +\infty} \sum_{i=0}^N \frac{1}{\lambda^i} \; \vphi\circ f^i \rightarrow G
\]
uniformly to a H\"older continuous $G\in\psh(\PP^k, \omega)$, where $f^i = f\circ \cdots \circ f$ means the $i$-th iterate of $f$. Observe that $G\leq 0$.
The Green current $T$ is defined by
\[\notag
T =\omega_G = \omega + dd^c G
\]
and satisfies $f^*T=\lambda T$.

Now, we briefly recall the notion of non-pluripolar product. 
Given $\psi\in\psh(\PP^k, \omega)$, let $\psi_j:=\max\{\psi, -j\}\in \psh(\PP^k, \omega)$. Define a measure by
\begin{displaymath}\notag
\mu_\psi:=\lim_{j\to +\infty}\mathbf{1}_{\{\psi>-j\}}(\omega+dd^c\psi_j)^k.
\end{displaymath}
The following class was introduced in Section 1 in \cite{GZ07}.
	\begin{displaymath}\notag
	\cE(\PP^k, \omega)=\left\{\psi\in\psh(\PP^k, \omega): \mu_\psi(\PP^k)=\int_{\PP^k}\omega^k \right\}.
	\end{displaymath}
	Equivalently, one can define $\psi\in\cE(\PP^k, \omega)$ if and only if $(\omega+dd^c\psi_j)^k(\{\psi\leq -j\})\to 0$ as $j\to\infty$.
In this paper, for a technical reason, we consider the following subclass:
\begin{displaymath}\notag
\cE^1(\PP^k, \omega):=\{\psi\in\cL^1(\langle\omega_\psi^k\rangle)\}\cap\cE(\PP^k, \omega).
\end{displaymath}
For $\psi\in\psh(\PP^k, \omega)$, the $p$-fold non-pluripolar product denoted by $\langle\omega_\psi^p\rangle$ is defined to be
	\begin{align*}
	\langle \langle\omega_\psi^p\rangle, R\rangle=\lim_{j\to\infty}\langle \omega_{\psi_j}^p, R\rangle
	\end{align*}
	for a smooth test form $R$ of bidegree $(k-p, k-p)$.
Here, since $\psi_j$ is bounded, $\omega_{\psi_j}^p$ is well defined in the sense of Bedford-Taylor. In particular, if the quasi-potential $\psi$ is bounded, then the non-pluripolar product coincides with the wedge product in the sense of Bedford-Taylor.

We will need the following properties of the class $\cE(\PP^k,\omega)$ introduced by Guedj-Zeriahi \cite[Theorem~2.6, Propostion~2.10]{GZ07} which generalised Cegrell's works in the local setting. 

\begin{thm} 
	\label{thm:BT-convergence} 
	Let $\psi \in \cE(\PP^k,\omega)$ and $\psi_j := \max\{\psi, -j\}$ be the canonical approximation of $\psi$. Then, for all Borel sets $B\subseteq \PP^k$, 
	$${\bf 1}_B (- |\psi_j|)(\omega + dd^c \psi_{j})^k \rightarrow {\bf 1}_B (- |\psi|) (\omega+ dd^c \psi)^k$$ 
	in the weak sense of measures.
\end{thm}

\begin{thm} 
	\label{thm:cegrell-ineq}
	There exist a universal constant $C>0$ such that for all $0\geq \psi_0, ..., \psi_k \subset \psh(\PP^k,\omega) \cap \cL^\infty(\PP^k)$,
	\[\notag
	0\leq \int_{\PP^k} (-\psi_0) \omega_{\psi_1} \wed \cdots \wed \omega_{\psi_k} \leq C \max_{0\leq j \leq k} \left(\int_X (-\psi_j) \omega_{\psi_j}^k\right)
	\]
\end{thm}

We show that the non-pluripolar product operator commutes with the pull-back operator $f^*$. 
\begin{lem}  We have for $1\leq p \leq k$,
\[\notag
	 f^{*} \left< \omega_u^p \right> = \left< (f^*\omega + dd^c u\circ f)^p\right>.
\]
\end{lem}

\begin{proof} 
Since $f$ is an endomorphism of $\PP^k$, 
$
	f^*\omega + dd^c u\circ f 
$ is a closed positive $(1,1)$-current. Hence, it follows from Boucksom-Eyssidieux-Guedj-Zeriahi \cite{BEGZ10} that  the left hand side is a well-defined closed positive $(p,p)$-current. By Dinh-Sibony \cite{DS07}, the pull-back operator is continuous:
\[\notag\begin{aligned}
	f^*\left< \omega_u^p \right> 
&= 	\lim_{L \to +\infty} f^* (\omega + dd^c \max\{u, -L\})^p \\
&=	\lim_{L \to +\infty} \left( f^*(\omega + dd^c \max\{u,-L\} \right)^p\\
&=	\lim_{L \to +\infty} \left( f^*\omega + dd^c \max\{u\circ f,-L\} \right)^p
\end{aligned}\]
Thus, we have finished the proof.
\end{proof}

In general, the non-pluripolar products do not admit bounded super-potentials. 
More precise statement is as follows.
\begin{prop}\label{prop:ex}
	For any analytic subset $H\subseteq \PP^k$ of pure codimension $1<l<k$, there exists a $\varphi\in\cE^1(\PP^k, \omega)$ such that  the super-potentials of $\langle\omega_\varphi^{k-l+1}\rangle$ have as their value $-\infty$ at $[H]$.
\end{prop}

\begin{proof}
	Due to Chow's theorem, let $\displaystyle H=\cap_{i=1}^j\{g_i=0\}$ for some homogeneous polynomials $g_i$ of degree $d_i$. Let $D:=d_1\cdots d_j$. Let $h$ be the $\omega$-psh function defined by 
	$$ h:=\frac{1}{D}\log\left|\frac{\sum_{i=1}^{j}|g_i|^{D/d_i}}{\|z\|^{D}}\right|\in\psh(\PP^k, \omega).$$
	By subtracting a constant, we may assume that $h\leq-1$. Due to Example 2.14 in \cite{GZ07} and Theorem \ref{thm:cegrell-ineq}, one can find a very small $s>0$ such that $\varphi:=-(-h)^s$ satisfies that its complex Monge-Amp\`ere operator $\langle\omega_\varphi^k\rangle$ is well defined and $\varphi\in \cL^1(\langle\omega_\varphi^k\rangle)$. Also, for every $1\leq i\leq k$, $\langle\omega_\varphi^{i}\rangle$ is well defined. 
	
	In the proof, we consider super-potentials of mean $0$. Note that the super-potentials on $\PP^k$ are everywhere well defined once we admit $-\infty$ as its value. 
%
	We consider $\langle\omega_\varphi^{k-l+1}\rangle$. Due to the continuity of super-potentials on $\PP^k$ with respect to the averaging, we have
	\begin{displaymath}
	\cU_{\langle\omega_\varphi^{k-l+1}\rangle}([H])=\lim_{\theta\to 0}\cU_{\langle\omega_\varphi^{k-l+1}\rangle}([H]_\theta)
	\end{displaymath}
	where $[H]_\theta$ is the regularization of $[H]$ by averaging in Section 2 in \cite{DS09}. For the above limit, see Corollary 3.17 in \cite{DS09}. Since $[H]_\theta$ is smooth, one can write
	\begin{displaymath}
	\cU_{\langle\omega_\varphi^{k-l+1}\rangle}([H]_\theta)=\langle\varphi\sum_{i=0}^{k-l}\omega^i\wedge\langle\omega_\varphi^{k-l-i}\rangle, [H]_\theta \rangle + c
	\end{displaymath}
	where $c>0$ is a constant independent of $\theta$. Let $M>0$ be an arbitrary large positive number. Then, for all sufficiently small $\theta>0$, we have $\supp [H]_\theta\subseteq \{\varphi<-M\}$. Then, since the mass of each $\omega^i\wedge\langle\omega_\varphi^{k-l-i}\rangle\wedge[H]_\theta$ is $1$, we have
	\begin{displaymath}
	\cU_{\langle\omega_\varphi^{k-l+1}\rangle}([H])=\lim_{\theta\to 0}\cU_{\langle\omega_\varphi^{k-l+1}\rangle}([H]_\theta)<-(k-l+1)M+c.
	\end{displaymath}
	Since this is true for arbitrary $M>0$, the super-potential of $\langle \omega_\varphi^{k-l+1}\rangle$ of mean $0$ is $-\infty$ at $[H]$.
\end{proof}

\section{Proof of Theorem~\ref{thm:main}}

In this section we will prove Theorem~\ref{thm:main}. More precisely, we are going to prove that for a smooth test form $\psi$ of bidegree $(k-p, k-p)$, there exists a constant $C>0$ independent of $\psi$ and $n$ such that
\[\notag\label{eq:exponential-convergence}
	\int_{\PP^k} \left(\lambda^{-pn} (f^n)^* \left<\omega_u^p\right> - T^p\right)\wed \psi \leq C\|\psi\|_{\cC^2} \lambda^\frac{-n}{2^k}.
\]
Using properties of the non-pluripolar product we have
\[\notag\begin{aligned}
&	\int_{\PP^k} \left(\lambda^{-pn} (f^n)^* \left<\omega_u^p\right> - T^p\right)\wed \psi \\
&=	\lim_{L\to +\infty} \int_{\PP^k} \left[\lambda^{-pn} (f^n)^* (\omega + dd^c \max\{u, -L\})^p - T^p\right]\wed \psi.
\end{aligned}\]
Therefore,  we may assume that $u$ is bounded and prove a uniform estimate with respect to $\|u\|_\infty$ of the integral on the right hand side.
\bigskip

Now we make preparations to achieve the goal. Let $u\in\psh(\PP^k, \omega)$ and bounded such that $\sup_{\PP^k}u=0$. Denote
\[\notag
	\tilde u := u -G - \sup_{\PP^k} (u-G).
\]
It follows from the continuity of the Green quasi-potential $G$ that for some $C_0>0$, we have
\begin{equation}\notag
	u - C_0 < \tilde  u < u+ C_0.
\end{equation}
We also have
\[ \label{eq:pull-back-S}
\begin{aligned}
	\frac{1}{\lambda} f^* \omega_u 
&= 	\frac{1}{\lambda} f^*\omega + \frac{1}{\lambda} dd^c u\circ f \\
&=	T + \frac{1}{\lambda} dd^c \tilde u \circ f. 
\end{aligned}\]
Inductively, using the fact that $(1/\lambda)  f^*T =T$,
\[\notag
	\frac{1}{\lambda^n} (f^n)^* \omega_u = T +\frac{1}{\lambda^n} dd^c \tilde u \circ f^n.
\]
For simplicity we write 
\[\notag
	v_n:= \tilde u\circ f^n, \quad S_n:= \frac{1}{\lambda^n} (f^n)^* \omega_u.
\]
Notice that from our running assumption, $v_n$ is also bounded.
The following proposition is the key estimate. 
\begin{prop}
\label{prop:each-term} 
Let  $1\leq i \leq p$. Then, there exists a uniform constant $C>0$ such that
\[\notag
	\int_{\PP^k} \frac{(-v_n)}{\lambda^n}  \cdot  S_n^{i-1}\wed T^{p-i} \wedge \omega^{k-p+1}  \leq C \left(\frac{1}{\lambda^n}\right)^\frac{1}{2^{k-p+1}}
\]
for every $n\geq 1$.
\end{prop}

We first prove several lemmas. 

\begin{lem} 
\label{lem:bound-gradient-G}
We have for $0\leq i \leq p-1$
\[\notag\begin{aligned}
	\int_{\PP^k} dG \wed d^c G \wed S_n^i\wed T^{p-i-1}\wed \omega^{k-p} 
&\leq		\|G\|_\infty,
\end{aligned}\]
where $\|\cdot\|_\infty := \sup_X |\cdot|$ denotes the uniform norm.
\end{lem}

\begin{proof}
Set
\[\notag
	Z:= S_n^i\wed T^{p-i-1}\wed \omega^{k-p}.
\]
Then, by integration by parts,
\[\notag\begin{aligned}
	\int_{\PP^k} dG \wed d^c G \wed Z 
&=	 \int_{\PP^k} - G dd^c G \wed Z \\
&=	\int_{\PP^k} -G  \; Z \wed T + \int_{\PP^k} G\;  Z \wed  \omega
\end{aligned}\]
Since $G\leq 0$ and it is H\"older continuous on $X$, we have
\[\notag\begin{aligned}
\int_{\PP^k} dG \wed d^c G \wed Z 
&\leq		\|G\|_\infty \int_{\PP^k} Z \wed T \\
&= \|G\|_\infty \int_{\PP^k} \omega^k.
\end{aligned}\]
The last equality follows from a cohomological argument.
We complete the proof.
\end{proof}

We will use the following result inductively to get the proposition.

\begin{lem}
\label{lem:induction-step}  We have for $0\leq i \leq p-1$,
\[\notag\begin{aligned}
&	\int_{\PP^k} \frac{(-v_n)}{\lambda^n} S_n^i\wed T^{p-i-1}\wed \omega^{k-p+1} \\
&\leq		\int_{\PP^k} \frac{(-v_n)}{\lambda^n} \wed S_n^i\wed T^{p-i}\wed \omega^{k-p} \\
&\quad + \|G\|_\infty^{1/2}\left(\int_{\PP^k} \frac{(-v_n)}{\lambda^n}  S_n^{i+1}\wed T^{p-i-1}\wed \omega^{k-p} \right)^\frac{1}{2}.
\end{aligned}\]
\end{lem}

\begin{proof}
As previously, we set $Z:= S_n^i\wed T^{p-i-1}\wed \omega^{k-p}.$
Then,
\[\label{eq:sum-ab}
\begin{aligned}
	\int_{\PP^k} (-v_n) Z \wed  \omega
= 	\int_{\PP^k} (-v_n) T \wed Z +	 \int_X v_n dd^c G \wed Z. 
\end{aligned}\]
The second term of the right hand side, by integration by parts and Cauchy-Schwarz's inequality, is estimated as follows:
\[\label{eq:sum-b}\begin{aligned}
	\left|\int_{\PP^k} v_n dd^c G \wed Z\right| 
&= 	\left|\int_{\PP^k} dv_n \wed d^c G  \wed Z\right|  \\
&\leq		\left(\int_{\PP^k} dv_n \wed d^c v_n \wed Z  \right)^\frac{1}{2}
	\left(\int_{\PP^k} dG \wed d^c G \wed Z\right)^\frac{1}{2} \\
&=		\left(\int_{\PP^k} (-v_n) dd^c v_n \wed Z\right)^\frac{1}{2} \left(\int_{\PP^k} dG \wed d^c G \wed Z \right)^\frac{1}{2}. 
\end{aligned}
\]
Lemma~\ref{lem:bound-gradient-G} implies that the last factor on the right hand side is bound by $\|G\|_\infty^{1/2}$. We now focus on the first integral on the right hand side. From \eqref{eq:pull-back-S}, we have
\[\label{eq:sum-b1}
\begin{aligned}
	\int_{\PP^k} (-v_n) dd^c v_n \wed Z  = 	
	\int_{\PP^k} (-v_n) (f^n)^* \omega_u \wed Z  + \int_{\PP^k} v_n (f^n)^* T \wed Z.
\end{aligned}\]
Recall that
$
	v_n = \tilde u\circ f^n \leq 0.
$
We conclude that
\[\label{eq:sum-b0}
	\int_{\PP^k}(-v_n) dd^c v_n \wed Z \leq \int_{\PP^k} (-v_n) \cdot (f^n)^*\omega_u \wed Z.
\]
Combining \eqref{eq:sum-ab}, \eqref{eq:sum-b}, \eqref{eq:sum-b1} and \eqref{eq:sum-b0} we have
\[\notag\label{eq:sum-cd} \begin{aligned}
	\int_{\PP^k} (-v_n) Z\wed \omega
&\leq		\int_{\PP^k} (-v_n) T \wed Z \\
&\quad + \|G\|_\infty^{1/2}\left(\int_{\PP^k} (-v_n) (f^n)^*\omega_u  \wed Z \right)^\frac{1}{2}. 
\end{aligned}\]
Dividing both sides of the inequality by $\lambda^n$, we get the desired inequality.
\end{proof}

For notational convenience, we write
\[\notag
F(\alpha, \beta, \gamma):= \int_{\PP^k} \frac{(-v_n)}{\lambda^n} S_n^\alpha\wed T^\beta\wed \omega^\gamma.
\]
From the invariance of $T$, we have the following proposition.
\begin{prop}\label{prop:initial_measure_estimate}
	For a positive Radon measure $\mu$ on $\PP^k$, we have
	\[\label{eq:pull-back-measure}
	\int_{\PP^k} f^* (d\mu) = \lambda^k \int_{\PP^k} d\mu.
	\]
	In particular, there is a constant $C_1>0$ independent of $n$ such that for every $1\leq i\leq p-1$ and $0\leq j \leq k-p+1$,
	\[\notag\begin{aligned}
	F(i+j, k-i-j,0) = 	\frac{1}{\lambda^n} \int_{\PP^k} (-\tilde u) \omega_{u}^{i+j} \wed T^{k-i-j} 	\leq		\frac{C_1}{\lambda^n}
	\end{aligned}\]
\end{prop}

\begin{proof} The first assertion is well known. Since $f^* T = dT$, we can write
	\[\notag\begin{aligned}
	&	\int_{\PP^k} \frac{(-v_n)}{\lambda^n} S_n^{i+j} \wed T^{k-i-j} \\
	&= 	\int_{\PP^k} \frac{-\tilde u \circ f^n}{\lambda^n} \left(\frac{1}{\lambda^n} (f^n)^*\omega_u\right)^{i+j} \wed T^{k-i-j} \\
	&=	\int_{\PP^k} \frac{ (f^n)^* (-\tilde u)}{\lambda^n} \left(\frac{1}{\lambda^n} (f^n)^*\omega_u\right)^{i+j} \wed \left(\frac{1}{\lambda^n} (f^n)^*T\right)^{k-i-j} \\
	&=	\frac{1}{\lambda^n}\int_{\PP^k} (-\tilde u) \omega_u^{i+j} \wed T^{k-i-j},
	\end{aligned}\]	
	where we used \eqref{eq:pull-back-measure} for the third equality. Furthermore, 
	using the fact that $u \in \cE^1(\PP^k,\omega)$ and Theorem~\ref{thm:cegrell-ineq} we get that the last integral on the right hand side is bounded by
	\[\notag
	C_1:=C \max \left(\int_{\PP^k} (-\tilde u) \omega_u^k, \int_{\PP^k} (-G) T^k \right).
	\]
	for some $C>0$. Therefore,
	\[\notag
	F(i+j, k-i-j,0) \leq \frac{C_1}{\lambda^n}.
	\]
	Thus, we are done.
\end{proof}

We use Lemma \ref{lem:induction-step} to estimate the integral of the form $F$ with $\gamma=k-p+1$ in terms of $F$'s with $\gamma=0$.
\begin{cor} \label{cor:bound-final}
There exists a constant $C>0$ such that for every $1\leq i\leq p-1$ and $0\leq j \leq k-p+1$,
\[\notag
	F(i,p-i-1,k-p+1) \leq		
	C \sum_{j=0}^{k-p+1} \left[ F(i+j, k-i-j,0)\right]^\frac{1}{2^j}.
\]
\end{cor}

\begin{proof}
Lemma~\ref{lem:induction-step} gives us that
\[\notag
	F(i,p-i-1,k-p+1) \leq F(i, p-i, k-p) + \|G\|_\infty^{1/2} \left[F(i+1, p-i-1, k-p)\right]^\frac{1}{2}.
\]
Observe that the third argument $\gamma$ in $F(\alpha, \beta, \gamma)$ decreases by $1$. Repeating this until $\gamma$ becomes $0$, we finally arrive at
\[\notag
	F(i,p-i-1,k-p+1)\leq C \sum_{j=0}^{k-p+1} \left[ F(i+j, k-i-j,0)\right]^\frac{1}{2^j}
\]
for some $C>0$. Thus, we have finished the proof.
\end{proof}

Now, we are ready to conclude the proposition and the theorem.
\begin{proof}[Proof of Proposition~\ref{prop:each-term}]
By Proposition \ref{prop:initial_measure_estimate} and Corollary~\ref{cor:bound-final}, for some uniform constants $C>0$ and $C'>0$, we have
\[\notag\begin{aligned}
	F(i,p-i-1, k-p+1) 
&\leq 	C \sum_{j=0}^{k-p+1}\left( \frac{1}{\lambda^n}\right)^\frac{1}{2^{j}} \\
&\leq		C' \left( \frac{1}{\lambda^n}\right)^\frac{1}{2^{k-p+1}}.
\end{aligned}\]
This is what we are after.
\end{proof}

\begin{proof}[End of proof of Theorem~\ref{thm:main}] Let $\psi$ be a smooth of bidegree $(k-p,k-p)$. 
We need to show that
\[\notag\label{eq:ex-co-b}
	\int_{\PP^k} \left[ \lambda^{-pn} (f^n)^* \omega_u^p - T^p\right]\wed \psi \leq C' \|\psi\|_{\cC^2}\lambda^\frac{-n}{2^k}
\]
for an uniform constant $C'>0$ independent of $\|u\|_\infty$. Recall that we are assuming that $u$ is bounded. Telescoping series yields
\[\notag
	\int_{\PP^k} \left[ \lambda^{-pn}(f^n)^*\omega_u^p -T^p \right]\wed \psi = \frac{1}{\lambda^n}  \int_{\PP^k} \sum_{i=1}^{p-1}dd^c v_n \wed S^i \wed T^{p-i} \wed  \psi
\]
We have by integration by parts, 
\[\notag\begin{aligned}
\frac{1}{\lambda^n}  \int_{\PP^k} \sum_{i=0}^{p-1}dd^c v_n \wed S^i \wed T^{p-i-1} \wed  \psi = \int_{\PP^k} \frac{v_n}{\lambda^n} \sum_{i=0}^{p-1} S^i \wed T^{p-i-1} \wed dd^c\psi.
\end{aligned}\]
Since $dd^c \psi$ is a smooth $(k-p+1, k-p+1)$-form. There is a positive constant $C_2>0$ such that
\[\notag
	- C_2 \|\psi\|_{\cC^2}\omega^{k-p+1} \leq dd^c\psi \leq C_2 \|\psi\|_{\cC^2} \omega^{k-p+1}.
\]
Therefore, from the negativity of $v_n$, we have
\[\notag\begin{aligned}
&	0\geq \int_{\PP^k} \frac{v_n}{\lambda^n} \sum_{i=1}^{p-1} S^i \wed T^{p-i} \wed dd^c\psi \\
&\geq 	2C_2 \|\psi\|_{\cC^2} \int_{\PP^k} \frac{v_n}{\lambda^n} \sum_{i=0}^{p-1} S^i \wed T^{p-i-1} \wed \omega^{k-p+1}
\end{aligned}\]
It follows from Propostion~\ref{prop:each-term} that the last integral is bound from below by
\[\notag
	 -2CC_2 \|\psi\|_{\cC^2} \left(\frac{1}{\lambda^n}\right)^\frac{1}{2^{k-p+1}}
\]
where the constant $C>0$ is from Proposition \ref{prop:each-term}. So, we have finished the proof of the theorem.
\end{proof}

\end{document}